\documentclass[12pt]{amsart}
\usepackage{amssymb}
\usepackage{mathtools}
\usepackage{epsfig}
\usepackage{amscd}
\usepackage{amsmath}
\usepackage{graphicx}
\usepackage{wasysym}
\usepackage{enumerate}
\usepackage{ifpdf}
\usepackage{amsfonts}
\usepackage{amsthm}
\usepackage[hyperfootnotes=false]{hyperref}
\hypersetup{hidelinks}
\usepackage{setspace}
\usepackage[usenames]{xcolor}
\usepackage{tikz}

\newtheorem*{theorem*}{Theorem}

\numberwithin{equation}{section}

\def\Re{{\sf Re}\,}

\def\1#1{\overline{#1}}
\def\2#1{\widetilde{#1}}
\def\3#1{\widehat{#1}}
\def\4#1{\mathbb{#1}}
\def\5#1{\frak{#1}}
\def\6#1{{\mathcal{#1}}}

\newcommand{\R}{\mathbb R}

\newcommand{\C}{\mathbb C}

\newcommand{\B}{\mathbb B}

\newcommand{\Hol}{{\sf Hol}}

\newcommand{\D}{\mathbb D}

\newcommand{\N}{\mathbb N}

\def\Re{{\sf Re}\,}

\newcommand{\abs}[1]{\left|#1\right|}
\newcommand{\norm}[1]{\left\|#1\right\|}

\theoremstyle{theorem}

\def\id{{\sf id}}

\def\Re{{\sf Re}\,}

\newtheorem{theorem}{Theorem}[section]
\newtheorem{lemma}[theorem]{Lemma}
\newtheorem{proposition}[theorem]{Proposition}
\newtheorem{corollary}[theorem]{Corollary}

\theoremstyle{definition}
\newtheorem{definition}[theorem]{Definition}
\newtheorem{example}[theorem]{Example}

\theoremstyle{remark}
\newtheorem{remark}[theorem]{Remark}

\newtheorem{question}[theorem]{Question}
\numberwithin{equation}{section}

\title[Homeomorphic extension of quasi-isometries]{Homeomorphic extension of quasi-isometries for convex domains in $\mathbb C^d$ and iteration theory}

\author[F. Bracci]{Filippo Bracci$^\dag$
}
\address{F. Bracci: Dipartimento di Matematica, Universit\`a di Roma ``Tor Vergata", Via della Ricerca Scientifica 1, 00133, Roma, Italia.} \email{fbracci@mat.uniroma2.it}

\author[H. Gaussier]{Herv\'e Gaussier$^\ddag$
}
\address{H. Gaussier: Univ. Grenoble Alpes, CNRS, IF, F-38000 Grenoble, France}
\email{herve.gaussier@univ-grenoble-alpes.fr}

\author[A. Zimmer]{Andrew Zimmer$^*$
}
\address{A. Zimmer: Department of Mathematics, Louisiana State University, Baton Rouge, LA USA }
\email{amzimmer@lsu.edu}

\keywords{boundary extension; Gromov hyperbolicity; commuting maps}

\thanks{$^\dag\,$Partially supported by the MIUR Excellence Department Project awarded to the  
Department of Mathematics, University of Rome Tor Vergata, CUP E83C18000100006}
\thanks{$^\ddag\,$Partially supported by ERC ALKAGE}
\thanks{$^*\,$Partially supported by the National Science Foundation under grants DMS-1760233 and DMS-1904099.}

\begin{document}

\begin{abstract}
We study the homeomorphic extension of biholomorphisms between convex  domains in $\mathbb C^d$ without boundary regularity and boundedness assumptions. Our approach relies on methods from coarse geometry,  namely the correspondence between the Gromov boundary and the topological boundaries of the domains and the dynamical properties of commuting 1-Lipschitz maps in Gromov hyperbolic spaces. This approach not only allows us to prove extensions for biholomorphisms, but for more general quasi-isometries between the domains endowed with their Kobayashi distances. 
\end{abstract}

\maketitle

\sloppy
\section{Introduction and results}
The aim of the paper is to investigate boundary extension of biholomorphisms, and more generally of quasi-isometries, between domains in the complex Euclidean space $\mathbb C^d$, $d\geq 1$, under some (geo)metric assumptions on the domains, regardless of boundary regularity or boundedness of the domains.

The Fefferman extension theorem \cite{Fe74} states that every biholomorphism between bounded strongly pseudoconvex domains with $C^\infty$ boundaries extends as a $C^\infty$ diffeomorphism to the closures of the domains. This seminal result reduces the equivalence problem between such domains to the comparison of CR invariants of the boundaries of the domains.  In the case where the domains are assumed neither smooth nor bounded, the question of homeomorphic extension of biholomorphisms seems quite difficult to attack with methods from Complex Analysis and Geometry.

Motivated by a result of Balogh and Bonk~\cite{BaBo00}, we consider this problem from a coarse geometry point of view. Balogh and Bonk proved that the Kobayashi distance on a bounded strongly pseudoconvex domain is Gromov hyperbolic and the Gromov boundary coincides with the Euclidean boundary. For Gromov hyperbolic metric spaces, it is a well known fact that homeomorphic quasi-isometries extend as homeomorphisms between the Gromov compactifications of the metric spaces, see~\cite{GhHa}. Since every biholomorphism between two domains is an isometry when the domains are endowed with their Kobayashi distances, this provides a new proof that every biholomorphism between strongly pseudoconvex domains extends to a homeomorphism of the Euclidean closures. Although this conclusion is weaker than Fefferman's result, it holds for a much larger class of maps - those that are quasi-isometries relative to the Kobayashi distances. 

The Gromov boundary of a Gromov hyperbolic metric space has a quasi-conformal structure which is preserved by isometries (see for instance~\cite[Section 3]{KB2002}). Recently Capogna and Le Donne~\cite{CL2017} used this quasi-conformal structure and results about the regularity of conformal maps between sub-Riemannian manifolds to provide a new proof of Fefferman's smooth extension theorem. Thus for bounded strongly pseudoconvex domains the extension theory of biholomorphisms can be completely understood using the theory of Gromov hyperbolic metric spaces.

These results motivate the following question:

\begin{question}\label{question:one} For which domains $D \subset \mathbb{C}^d$ is the Kobayashi distance $K_D$ a complete Gromov hyperbolic distance? 
\end{question}

To apply the theory of Gromov hyperbolicity to the problem of extensions of biholomorphisms, one also needs to identify the Gromov boundary with the Euclidean boundary. For unbounded domains in $\mathbb{C}^d$ there are several natural choices of an Euclidean compactification, but it seems like the right one is the end compactification of the closure  (see Section \ref{Sec:5}). 

\begin{definition} Given a domain $D \subset \mathbb{C}^d$ the \emph{Euclidean end compactification of $D$}, denoted by $\overline{D}^{\star}$, is defined to be the end compactification of $\overline{D}$.
\end{definition} 

\begin{remark} \ 
\begin{enumerate}
\item When $D\subset \mathbb C^d$ is an unbounded convex domain, the end compactification of $\overline{D}$ has either one or two points ``at infinity'' (see Section \ref{Sec:5}). 
\item It is important that we are taking the end compactification of $\overline{D}$ instead of $D$; if $D \subset \C^d$ is a convex domain, then the end compactification of $D$ is just the one point compactification of $D$. 
\end{enumerate}
\end{remark}

One can then ask:

\begin{question} For which classes of domains $D \subset \mathbb{C}^d$ is the following true: if the Kobayashi distance $K_D$ on $D$ is Cauchy complete and Gromov hyperbolic, then the identity map $\id : D \rightarrow D$ extends to a homeomorphism $\overline{\id}: \overline{D}^{\star} \rightarrow \overline{D}^{G}$ where $\overline{D}^G$ is the Gromov compactification of the metric space $(D,K_D)$?
\end{question}

In dimension one, for a simply connected domain $D \subset \C$ different from $\C$,  the Riemann mapping theorem implies that $(D,K_D)$ is isometric to the real hyperbolic 2-space and hence is Gromov hyperbolic. Further, the Gromov compactification is equivalent to Carath\'eodory's prime ends topology and thus a complete answer to Question~1.3 is given by the Carath\'eodory extension theorem: if $D\subsetneq \C$ is a simply connected domain, the identity map $\id : D \rightarrow D$ extends to a homeomorphism $\overline{\id}: \overline{D}^{\star} \rightarrow \overline{D}^{G}$ if and only if $D$ is a Jordan domain. 

Since every convex  domain in $\C$ (different from $\C$) is a Jordan domain,  the previous question has positive answer for one-dimensional convex domains different from $\C$. Note also that every convex domain in $\C$ which is different from $\C$ is  Gromov hyperbolic with respect to its (complete) Kobayashi distance. 

One aim of this paper is to extend such a result about convex domains to higher dimension, assuming the natural hypotheses suggested by the one dimensional case.

 A convex domain is called \emph{$\mathbb C$-proper} if it does not contain any complex affine lines. For such domains we prove the following. 

\begin{theorem}\label{ext-thm}
Let $D$ be a $\mathbb C$-proper convex domain in $\C^d$. If $(D,K_D)$ is  Gromov hyperbolic, then the identity map $\id : D \rightarrow D$ extends to a homeomorphism $\overline{\id}: \overline{D}^{\star} \rightarrow \overline{D}^{G}$.
\end{theorem}

\begin{remark}\  \begin{enumerate}
\item Theorem~\ref{ext-thm} does not assume that $D$ is bounded or has smooth boundary. By a result of Barth~\cite{Bar80}, when $D$ is convex the Kobayashi distance is Cauchy complete if and only if  $D$ is (Kobayashi) hyperbolic if and only if $D$ is $\mathbb C$-proper (also see  \cite{BS2009}).
\item There are many examples of convex domains where the Kobayashi distance is Gromov hyperbolic, see~\cite{Zi16,Zi17}. 
\item For convex domains any of the classical invariant distances, such as the Bergman distance, the Carath\'eodory distance or the K\"ahler-Einstein distance, are bi-Lipschitz equivalent~\cite{Fr91}. Since Gromov hyperbolicity is a quasi-isometric invariant, this implies that the above theorem is also true for any of the other classical invariant distances. 
\end{enumerate}
\end{remark}

As a consequence of Theorem~\ref{ext-thm} we obtain an extension result for quasi-isometries between some domains in $\C^d$, see Section~\ref{Sec:2} for precise definitions.

\begin{corollary}\label{eucl-thm}
Let $D$ and $\Omega$ be domains in $\C^d$. We assume:

\begin{enumerate}
\item $D$ is either a bounded, $C^2$-smooth strongly pseudoconvex domain,  or a convex $\mathbb C$-proper domain, such that $(D,K_D)$ is Gromov hyperbolic,

\item $\Omega$ is convex.
\end{enumerate}

Then every quasi-isometric homeomorphism  $F: (D,K_D) \rightarrow (\Omega,K_{\Omega})$ extends as a homeomorphism $\overline{F}: \overline{D}^{\star} \rightarrow \overline{\Omega}^{\star}$. In particular, every biholomorphism $F:D \to \Omega$ extends as a homeomorphism $\overline{F}: \overline{D}^{\star} \rightarrow \overline{\Omega}^{\star}$.
\end{corollary}

In case $D$ is the unit ball (or more generally a strongly convex domain with smooth boundary), the previous extension result for biholomorphisms has been proved in \cite{BrGa17S} in case $\Omega$ is bounded, as an application of a Carath\'eodory prime ends type theory in higher dimension called ``horosphere topology'', and in  \cite{BrGa17} for the case $\Omega$ is unbounded by using a direct argument relying on the dynamics of semigroups of holomorphic self-maps. 

In Section~\ref{Sec:example}  we will provide examples showing that the hypotheses in Corollary~\ref{eucl-thm} are  optimal. 

\subsection{Iterating holomorphic maps} The proof of Theorem~\ref{ext-thm} is, quite surprisingly, based on dynamical properties of commuting semigroups of holomorphic self-maps, and it also provides new information about the dynamics of iterates of holomorphic maps. 

Given a bounded domain $D \subset \mathbb{C}^d$ and a holomorphic self map $f : D \rightarrow D$, Montel's theorem implies that the sequence of iterates of $f$, denoted by $\{f^n\}$, forms a relatively compact set in the space of holomorphic maps $D \rightarrow \overline{D}$. In particular, given a sequence $n_j \rightarrow \infty$ one can always find a subsequence $n_{j_k} \rightarrow \infty$ such that $f^{n_{j_k}}$  converges locally uniformly to a holomorphic map $g : D \rightarrow \overline{D}$.  Surprisingly, there are some cases where the behavior of the limits are independent of the subsequence chosen. This is demonstrated by the classical Denjoy-Wolff theorem:

\begin{theorem*}[Denjoy-Wolff \cite{D1926, W1926}]\label{thm:WD}
Let $f:\mathbb{D} \rightarrow \mathbb{D}$ be a holomorphic map. Then either:
\begin{enumerate}
\item $f$ has a fixed point in $\mathbb{D}$; or
\item there exists a point $\xi \in \partial \mathbb{D}$ so that
\begin{equation*}
\lim_{n \rightarrow \infty} f^n(x) = \xi
\end{equation*}
 for any $x \in \mathbb{D}$, this convergence being uniform on compact subsets of $\mathbb{D}$.
\end{enumerate}
\end{theorem*}

The Denjoy-Wolff theorem has been extended in the past by many authors in different situations (see, {\sl e.g.}, \cite{Ababook, AbRa, K2001} and references therein for a detailed account).  

The Kobayashi distance on $\mathbb{D}$ coincides, up to a constant, with the standard Poincar{\'e} distance and the metric space $(\mathbb{D}, K_{\mathbb{D}})$ is Gromov hyperbolic. Further,  any holomorphic map $f: \mathbb{D} \rightarrow \mathbb{D}$ is 1-Lipschitz relative to the Kobayashi distance. 

Karlsson proved the following abstract version of the Wolff-Denjoy theorem for general Gromov hyperbolic metric spaces (see Section~\ref{Sec:2} for details and definitions related to the Gromov boundary):

\begin{theorem*}[Karlsson, Prop. 5.1 in \cite{K2001}]\label{thm:gromov_wd} Suppose $(X,d)$ is a proper geodesic Gromov hyperbolic metric space and denote by $\partial_G X$ its Gromov boundary. If $f: X \rightarrow X$ is 1-Lipschitz, then either:
\begin{enumerate}
\item for every $p \in X$, the orbit $\{ f^n(x) : n \in \mathbb{N}\}$ is bounded in $(X,d)$, or
\item there exists a unique $\xi \in \partial_G X$ so that for all $x \in X$,
\begin{align*}
\lim_{n \rightarrow \infty} f^n(x) = \xi,
\end{align*}
in the Gromov compactification. 
\end{enumerate}
\end{theorem*}

Karlsson's Theorem 
together with Balogh and Bonk's result~\cite{BaBo00} proves in particular a Denjoy-Wolff theorem for bounded $C^2$-smooth strongly pseudoconvex domains in $\C^d$. Such a result was proved directly by Abate in \cite{Aba2}: if $D$ is a $C^2$-smooth strongly pseudoconvex domain and $f:D\to D$ is holomorphic, then either $\{f^n(z)\}$ is relatively compact in $D$ for every $z\in D$, or there exists a unique  point in $\partial D$ such that every orbit of $f$ converges to such a point. 
Huang in \cite{Hu}, under the assumption of $C^3$ boundary smoothness, proved later that if $D$ is a topological contractible bounded strongly pseudoconvex domain, then $f$ has a fixed point in $D$ if and only if there is a point $z \in D$ such that the orbit $\{f^n(z)\}$ is relatively compact in $D$. The $C^2$-smooth boundary case still remains open, and it is shown in \cite{AbHe} that Huang's result does not hold in general as soon as strict pseudoconvexity fails at just one boundary point. 

On the other hand, if $D\subset \C^d$ is a hyperbolic convex domain and $f:D\to D$ is holomorphic, then $f$ has no fixed points in $D$ if and only if every orbit of $\{f^n\}$ is compactly divergent (see \cite{Aba0, Aba1, KS, BS2009}). Therefore, as a direct corollary to Karlsson's result and Theorem~\ref{ext-thm} we have the following:

\begin{corollary}\label{cor:WD_convex}
Let $D \subset \mathbb{C}^d$ be a $\mathbb C$-proper convex domain such that $(D,K_D)$ is  Gromov hyperbolic. If $f: D \rightarrow D$ is holomorphic, then either: 
\begin{enumerate}
\item $f$ has a fixed point in $D$; or
\item there exists a point $\xi \in \overline{D}^{\star} \setminus D$, called the {\sl Denjoy-Wolff point} of $f$, so that
\begin{equation*}
\lim_{n \rightarrow \infty} f^n(x) = \xi
\end{equation*}
 for any $x \in D$, this convergence being uniform on compact subsets of $D$. In particular, either $\xi\in \partial D$ and $\lim_{n \rightarrow \infty} f^n(x) = \xi$ or $\lim_{n\to +\infty}\|f^n(x)\|=\infty$ for all $x\in D$. 
\end{enumerate}
\end{corollary}

In case $D$ is a bounded strictly convex domain, a Denjoy-Wolff theorem of the previous type has been proved by Budzy\'nska \cite{Bu}, while, for bounded $C^2$-smooth strictly $\C$-linearly convex domains, the result is due to Abate and Raissy \cite{AbRa}.

To the best our knowledge there are no prior results of Denjoy-Wolff type which hold for general classes of unbounded domains. 

\subsection{Commuting holomorphic maps} As mentioned above, the proof of Theorem~\ref{ext-thm} relies in an essential way on the study of commuting holomorphic maps in domains in $\mathbb C^d$. 

In 1973, Behan \cite{Be73} proved that two commuting holomorphic self-maps of the unit disc $\D$ with no fixed points in $\D$ either have to share the same Denjoy-Wolff point on $\partial \D$ or are hyperbolic automorphisms of $\D$. In~\cite{Br99, Br98} the first named author generalized Behan's result to the unit ball and to smooth bounded strongly convex domains, proving that if two commuting holomorphic maps  have distinct Denjoy-Wolff points, then the restrictions of the maps to the unique complex geodesic joining the two points are automorphisms of such a complex geodesic. The aim of the following theorem is to generalize Behan's result to commuting 1-Lipschitz maps in Gromov hyperbolic spaces:

\begin{theorem}\label{commut-thm}
Let $(X,d)$ be a proper geodesic Gromov hyperbolic metric space. Let $f,g: X \rightarrow X$ be commuting 1-Lipschitz maps. Suppose  there exist $\xi_f \neq \xi_g \in \partial_G X$ and $x_0 \in X$ such that
\begin{equation}\label{com-eq}
\lim_{n\to\infty}f^{n}(x_0) =\xi_f \ \ {\rm and } \ \ \lim_{n\to\infty}g^{n}(x_0) =\xi_g,
\end{equation}
in the Gromov compactification. Then there exist a totally geodesic closed subset $M \subset X$ and a $1$-Lipschitz map $\rho: X \rightarrow M$ such that:

\begin{enumerate}
\item $\rho \circ \rho = \rho$,
\item $f(M)=g(M)=M$ and $f|_{M}$ and $g|_{M}$ are isometries of $(M,d|_M)$.
\end{enumerate}
\end{theorem}

\begin{remark} Since $(X,d)$ is a proper geodesic Gromov hyperbolic metric space, in view of Karlsson's Theorem, the existence of a point $x_0 \in X$ such that $f^{n}(x_0)$ converges to $\xi_f$ is equivalent to the convergence of $f^{n}(x)$ to $\xi_f$, for all $x \in X$.
\end{remark}

As an application of Theorem \ref{commut-thm} and Balogh and Bonk's theorem, we have the following generalization of Behan's result:

\begin{corollary}\label{str-cor}
Let $D$ be a bounded, $C^2$-smooth strongly pseudoconvex domain in $\C^d$, and let $f,g$ be commuting holomorphic maps from $D$ to $D$. Suppose that there exist $p_f \neq p_g \in \partial D$ and a point $z_0 \in D$ such that
$$
\lim_{n \rightarrow \infty}f^{n}(z_0) = p_f \ \ {\rm and} \ \ \lim_{n \rightarrow \infty}g^{n}(z_0) = p_g.
$$
Then there exists a complex geodesic $\Delta$ for $D$, which is a holomorphic retract of $D$ such that $p_f, p_g\in \partial \Delta$, $f(\Delta)=\Delta$, $g(\Delta)=\Delta$ and $f|_\Delta, g|_\Delta$ are (hyperbolic) automorphisms of $\Delta$. In particular, $p_f$ ({\sl resp.} $p_g$) is a boundary fixed point in the sense of admissible limits in $D$ for $g$ ({\sl resp.} for $f$).
\end{corollary}

By \cite[Theorem 1]{Hu}, if  $D$ is a bounded $C^3$-smooth strongly pseudoconvex domain in $\C^d$, which is topologically contractible, and $f$ is a holomorphic self-map of $D$, then there exists a unique point $p_f\in \partial D$ such that $\{f^n\}$ converges uniformly on compacta of $D$ to the constant map $\zeta\mapsto p_f$ if and only if $f$ has no fixed points in $D$.
Therefore, the previous corollary, in case $\partial D$ is $C^3$-smooth and $D$ is topologically contractible, says that if $f,g$ are commuting holomorphic self-maps of $D$ with no fixed points in $D$, then either $f, g$ have the same Denjoy-Wolff point or $f$, $g$ are automorphisms of a  complex geodesic for $D$ joining the Denjoy-Wolff points of $f$ and~$g$.

On the other hand,  as a direct  consequence of  Corollary \ref{cor:WD_convex} and  Theorem \ref{commut-thm}, we have

\begin{corollary}\label{conv-cor}
Let $D$ be a $\mathbb C$-proper convex domain in $\C^d$ such that $(D, K_D)$ is Gromov hyperbolic and let $f,g$ be commuting holomorphic maps from $D$ to $D$. Suppose that $f$ and $g$ have no fixed points in $D$ and let $p_f\in \overline{D}^\star\setminus D$ ({\sl resp.} $p_g\in \overline{D}^\star\setminus D$) be the Denjoy-Wolff point of $f$ ({\sl resp.} of $g$). Then, either $p_f=p_g$ or  there exists a holomorphic retract $M$ of $D$, of complex dimension $1 \leq k \leq  d$,  such that $p_f, p_g \in \overline{M}^\star\setminus M$, $f(M)=g(M)=M$ and $f_{|M},\ g_{|M} \in {\rm Aut}(M)$. 
\end{corollary}

\subsection{Outline of the paper} The paper is organized as follows. In Section \ref{Sec:2} we recall important properties about the Gromov hyperbolic metric spaces, in our context. In Section~\ref{Sec:4}, we study dynamical properties of commuting 1-Lipschitz maps in Gromov hyperbolic spaces.  In Sections~\ref{Sec:5} and~\ref{Sec:6}, we study, for an unbounded, convex, Gromov hyperbolic domain, the correspondence between its Gromov compactification and the end compactification of its Euclidean closure. We prove Theorem~\ref{ext-thm} in Section~\ref{Sec:6}. In Section~\ref{Sec:7}  we prove Corollary~\ref{eucl-thm}, Theorem~\ref{commut-thm} and Corollary~\ref{str-cor}. Finally, in Section \ref{Sec:example} we provide some examples showing that our hypotheses are optimal.

\medskip

{\sl Acknowledgments.} The authors thank the referees for their very useful comments which improved the original version.

\section{The Kobayashi metric} 

 In this expository section we recall the definition of the Kobayashi metric. Given a domain $\Omega \subset \mathbb{C}^d$ the \emph{(infinitesimal) Kobayashi metric} is the pseudo-Finsler metric
\begin{align*}
k_{\Omega}(x;v) = \inf \left\{ \abs{\xi} : f \in \Hol(\mathbb{D}, \Omega), \ f(0) = x, \ d(f)_0(\xi) = v \right\}.
\end{align*}
By a result of Royden~\cite[Proposition 3]{Roy1971} the Kobayashi metric is an upper semicontinuous function on $\Omega \times \mathbb{C}^d$. In particular, if $\sigma:[a,b] \rightarrow \Omega$ is an absolutely continuous curve (as a map $[a,b] \rightarrow \mathbb{C}^d$), then the function 
\begin{align*}
t \in [a,b] \rightarrow k_\Omega(\sigma(t); \sigma^\prime(t))
\end{align*}
is integrable and we can define the \emph{length of $\sigma$} to  be
\begin{align*}
\ell_\Omega(\sigma)= \int_a^b k_\Omega(\sigma(t); \sigma^\prime(t)) dt.
\end{align*}
One can then define the \emph{Kobayashi pseudo-distance} to be
\begin{equation*}
 K_\Omega(x,y) = \inf \left\{\ell_\Omega(\sigma) : \sigma\colon[a,b]
 \rightarrow \Omega \text{ is abs. cont., } \sigma(a)=x, \text{ and } \sigma(b)=y\right\}.
\end{equation*}
This definition is equivalent to the standard definition using analytic chains by a result of Venturini~\cite[Theorem 3.1]{Ven1989}.

When $\Omega$ is a bounded domain, $K_\Omega$ is a non-degenerate distance. For general domains there is no known characterization of when the Kobayashi distance is proper, but for convex domains we have the following result of Barth.

\begin{theorem}\cite{Bar80}\label{thm:barth}
Suppose $\Omega$ is a convex domain. Then the following are equivalent:
\begin{enumerate}
\item $\Omega$ is $\mathbb{C}$-proper, 
\item $K_\Omega$ is a non-degenerate distance on $\Omega$, 
\item $(\Omega, K_\Omega)$ is a proper metric space, 
\item $(\Omega, K_\Omega)$ is a proper geodesic metric space. 
\end{enumerate}
\end{theorem}

\section{The Gromov compactification of a Gromov hyperbolic space}\label{Sec:2}

Let $(X,d)$ be a metric space and let $I \subset \mathbb R$ be an interval, endowed with the Euclidean metric. An isometry $\gamma: I \rightarrow X$ is called a \emph{geodesic}. If $I=[a,b]$, we call $\gamma$ a {\sl geodesic segment}, if $I=\mathbb R_{\geq 0}$, we call $\gamma$ a {\sl geodesic ray} and if $I=\mathbb R$, we call $\gamma$ a {\sl geodesic line}.

We recall that $(X,d)$ is:

\begin{enumerate}
\item \emph{proper} if every closed ball is compact in $X$,\
\item \emph{geodesic} if every two points $x_1,x_2 \in X$ can be joined by a geodesic segment.
\end{enumerate}

\vspace{1mm}
If $(X,d)$ is a geodesic metric space, a {\sl geodesic triangle} is the union of geodesic segments $\gamma_i:[a_i,b_i] \rightarrow X$, $i=1,2,3$, such that $a_i < b_i$ for every $i=1,2,3$ and $\gamma_1(b_1) = \gamma_2(a_2)$, $\gamma_2(b_2) = \gamma_3(a_3)$, $\gamma_3(b_3) = \gamma_1(a_1)$. The geodesic segments $\gamma_1$, $\gamma_2$ and $\gamma_3$ are called the sides of the triangle.

\begin{definition}\label{grom-def1}
A proper geodesic metric space $(X,d)$ is {\sl Gromov hyperbolic} if there exists $\delta \geq 0$ such that every geodesic triangle is $\delta$-thin, namely if each side of the triangle is contained in a $\delta$-neighborhood of the union of the two other sides.
\end{definition}

We assume for the rest of this subsection that $(X,d)$ is a proper geodesic Gromov hyperbolic metric space. Let $x_0\in X$.
Then let $\mathcal{G}_{x_0}$ denote the space of geodesic rays $\gamma:[0,+\infty) \rightarrow X$ such that $\gamma(0)=x_0$, endowed with the topology of uniform convergence on compact subsets of $[0,+\infty)$. We consider on $\mathcal{G}_{x_0}$ the equivalence relation $\sim$ defined by
$$
\gamma \mathcal \sim \lambda \Leftrightarrow \sup_{t \geq 0} d(\gamma(t),\lambda(t)) <+\infty.
$$

\begin{definition}\label{grom-def2} \ 
\begin{enumerate}[(i)]
\item The \emph{Gromov boundary} $\partial_G X$ of $X$ is defined as the quotient space $\mathcal{G}_{x_0}/\sim$ endowed with the quotient topology.
\item The \emph{Gromov closure} of $X$ is $\overline{X}^G:=X \cup \partial_G X$.
\end{enumerate}
\end{definition}

The choice of the base point $x_0$ is irrelevant. Indeed, given $x_0, x_1\in X$, there is a natural map $J:\mathcal{G}_{x_0}/\sim\,\to \mathcal{G}_{x_1}/\sim$ defined as follows. Let $[\gamma]\in \mathcal{G}_{x_0}/\sim$, where $\gamma:[0,+\infty)\to X$ is a geodesic ray such that $\gamma(0)=x_0$. For $n\in \N$, let $\eta_n:[0, R_n]\to X$ be a geodesic segment such that $\eta_n(0)=x_1$ and $\eta_n(R_n)=\gamma(n)$. Up to extracting subsequences, using Arzel\'a-Ascoli's theorem, we can assume that $\{\eta_n\}$ converges locally uniformly to a geodesic ray $\eta:[0,+\infty)\to X$ such that $\eta(0)=x_1$. We let $J([\gamma])=[\eta]$. Since $(X,d)$ is Gromov hyperbolic, one can easily see that $\eta$ is contained in a finite neighborhood of $\gamma$, hence the map $J$ is well defined and bijective (the inverse being the map $\mathcal{G}_{x_1}/\sim\to \mathcal{G}_{x_0}/\sim$ defined in a similar manner). By the same token, one can see that $J$ is a homeomorphism.

The set $\overline{X}^G$ has a natural topology making it a compactification of $X$ (see for instance~\cite[Chapter III.H.3]{BH1999}) and with this topology, $\overline{X}^G$ is first countable and Hausdorff.

To understand this topology we introduce some additional notation: given a geodesic ray $\sigma\in \mathcal{G}_{x_0}$ define ${\rm End}(\sigma)$ to be the equivalence class of $\sigma$ and given a geodesic segment $\sigma:[0,R] \rightarrow X$ such that $\sigma(0)=x_0$, define ${ \rm End}(\sigma) = \sigma(R)$.  Then $\xi_n \rightarrow \xi$ in $\overline{X}^G$ if and only if for every choice of geodesics $\sigma_n$ with $\sigma_n(0)=x_0$ and ${ \rm End}(\sigma_n)=\xi_n$ every subsequence of $\{\sigma_n\}_{n \in \mathbb N}$ has a subsequence which converges locally uniformly to a geodesic $\sigma$ with ${ \rm End}(\sigma)=\xi$.

\subsection{Geodesics} In this section we recall some basic properties of geodesics in a Gromov hyperbolic metric space. 

By our description of the topology of $\overline{X}^G$ one has the following observation. 

\begin{remark}\label{obs:limits_of_geod} Suppose that $(X,d)$ is a proper geodesic Gromov hyperbolic metric space. If $\sigma : \mathbb R \rightarrow X$ is a geodesic, then the limits 
\begin{align*}
\lim_{t \rightarrow -\infty} \sigma(t) \text{ and } \lim_{t \rightarrow +\infty} \sigma(t)
\end{align*}
both exist in $\overline{X}^G$ and are distinct. 
\end{remark} 

One more important property of Gromov hyperbolic metric spaces is that geodesics joining two points in the boundary ``bend'' into the space. More precisely: 

\begin{theorem}\label{thm:visible} Let $(X,d)$ be a proper geodesic Gromov hyperbolic metric space and let $x_0 \in X$. If $\xi,\eta \in \partial_G X$ and $V_\xi, V_\eta$ are neighborhoods of $\xi,\eta$ in $\overline{X}^G$ so that $\overline{V_\xi} \cap \overline{V_\eta} = \emptyset$, then there exists a compact set $K \subset X$ with the following property: if $\sigma: [a,b] \rightarrow X$ is a geodesic with $\sigma(a) \in V_\xi$ and $\sigma(b) \in V_\eta$, then $\sigma \cap K \neq \emptyset$. 
\end{theorem}  

For a proof see for instance~\cite[page 54]{BGS1985} or~\cite[page 294]{BH1999}. This result has the following corollary. 

\begin{corollary}\label{cor:Gromov_Product} Let $(X,d)$ be a proper geodesic Gromov hyperbolic metric space and let $x_0 \in X$. If $\xi,\eta \in \partial_G X$ and $V_\xi, V_\eta$ are neighborhoods of $\xi,\eta$ in $\overline{X}^G$ so that $\overline{V_\xi} \cap \overline{V_\eta} = \emptyset$, then there exists some $A \geq 0$ such that 
\begin{align*}
d(x,y) \geq d(x,x_0)+d(x_0,y) - A
\end{align*}
for all $x \in V_\xi$ and $y \in V_\eta$. 
\end{corollary}  

\begin{proof} Let $K$ be the compact set from Theorem~\ref{thm:visible}. Then let 
\begin{align*}
A = 2 \max\{ d(k,x_0) : k \in K\}.
\end{align*}

Now suppose that  $x \in V_\xi$ and $y \in V_\eta$. Let $\sigma: [a,b] \rightarrow X$ be a geodesic with $\sigma(a)=x$ and $\sigma(b)=y$. Then there exists some $t \in [a,b]$ such that $\sigma(t) \in K$. Then 
\begin{align*}
d(x,y) = d(x,\sigma(t)) + d(\sigma(t), y) \geq d(x,x_0) + d(x_0,y) -A.
\end{align*}
\end{proof}

\subsection{Quasi-geodesics, quasi-isometries, and the shadowing lemma}

\begin{definition}\label{quasi-def} Let $(X,d_X)$ and $(Y,d_Y)$ be metric spaces and let $A \geq 1,\ B \geq 0$.

\begin{enumerate}
\item If $I \subset \mathbb R$ is an interval, then a map $\gamma:I \rightarrow X$ is an {\sl $(A,B)$-quasi-geodesic} if for all $s,t \in I$:
\begin{equation}\label{quasi-eq}
\frac{1}{A}|t-s| - B \leq d_X(\gamma(s), \gamma(t)) \leq A |t-s| + B.
\end{equation}
If $I=[a,b]$ (resp. $I=\mathbb R_{\geq 0}$ or $I=\mathbb R$) we call $\gamma$ a quasi-geodesic segment (resp. quasi-geodesic ray or quasi-geodesic line).
\item A map $f:X \rightarrow Y$ is an {\sl$(A,B)$-quasi-isometry} if for all $x_1, x_2 \in X$:
$$
\frac{1}{A} d_X(x_1,x_2) - B \leq d_Y(f(x_1), f(x_2)) \leq A d_X(x_1,x_2) + B.
$$
\end{enumerate}
\end{definition}

\begin{remark} \ \begin{enumerate} 
\item Notice that an $(A,B)$-quasi-geodesic in $(X,d)$ is an $(A,B)$-quasi-isometry from $\left(I, | \cdot |\right)$, where $I$ is an interval of $\mathbb R$, to $(X,d)$.
\item When $f$ is a bijective quasi-isometry from $(X,d_X)$ to $(Y,d_Y)$, then $f^{-1}$ is also a quasi-isometry. 
\item If $f: (X, d_X) \rightarrow (Y,d_Y)$ is a quasi-isometry and a bijection between proper geodesic metric spaces, then $(X,d_X)$ is Gromov hyperbolic if and only if $(Y,d_Y)$ is Gromov hyperbolic, see~\cite{GhHa}.
\end{enumerate}
\end{remark}

\vspace{1mm}
Throughout the paper, we will say that a curve $\mathcal C$ in a metric space $(X,d)$ is an $(A,B)$- quasi-geodesic if there is some parametrisation $\gamma:I \rightarrow X$ of $\mathcal C$, where $I\subset \mathbb R$ is an interval, such that $\gamma$ satisfies Condition~(\ref{quasi-eq}).

We will use the following fact repeatedly.

\begin{theorem}[Shadowing Lemma]\label{thm:shadow_lemma} Suppose that $(X,d)$ is a proper geodesic Gromov hyperbolic metric space. For any $A  \geq  1$ and $B \geq 0$ there exists $R > 0$ such that: if $\gamma_1:[a_1,b_1] \rightarrow X$ and $\gamma_2:[a_2,b_2] \rightarrow X$ are two $(A,B)$-quasi-geodesic segments with $\gamma_1(a_1) = \gamma_2(a_2)$ and $\gamma_1(b_1) = \gamma_2(b_2)$, then 
\begin{align*}
\max \left\{ \max_{t \in [a_1,b_1]} d(\gamma_1(t), \gamma_2([a_2,b_2])) , \max_{t \in [a_2,b_2]} d(\gamma_2(t), \gamma_1([a_1,b_1])) \right\} \leq R.
\end{align*} 
\end{theorem} 

For a proof see for instance~\cite{CoDePa90} Th\'eor\`eme 1.3 p.25, or~\cite{GhHa} Th\'eor\`eme 11 p.87.

Using Theorem~\ref{thm:shadow_lemma} and Remark~\ref{obs:limits_of_geod} we have the following. 

\begin{remark}\label{obs:limits_of_quasi_geod} Suppose that $(X,d)$ is a proper geodesic Gromov hyperbolic metric space. If $\sigma : \mathbb R \rightarrow X$ is a quasi-geodesic, then the limits 
\begin{align*}
\lim_{t \rightarrow -\infty} \sigma(t) \text{ and } \lim_{t \rightarrow +\infty} \sigma(t)
\end{align*}
both exist in $\overline{X}^G$ and are distinct. 
\end{remark}

\section{Commuting 1-Lipschitz self maps of a Gromov hyperbolic metric space}\label{Sec:4}

In this section we study commuting 1-Lipschitz maps of a proper geodesic Gromov hyperbolic metric space.

Suppose that $(X,d)$ is a proper geodesic Gromov hyperbolic metric space. Further, suppose that $f,g : X \rightarrow X$ are commuting 1-Lipschitz maps and that there exist $\xi_f, \xi_g \in \partial_G X$ so that for all $x \in X$, it holds

\begin{equation}\label{different Wolff}
f^{n}(x)\rightarrow \xi_f \text{ and } g^{n}(x)\rightarrow \xi_g.
 \end{equation}

\begin{proposition}\label{prop:commuting} With the notation above, suppose that $\xi_g \neq \xi_f$. Then there exists a compact set $K \subset X$ such that: for every $m \geq 0$ there exists $n = n(m) \geq 0$ with 
$$
K \cap f^mg^{n}(K) \neq \emptyset.
$$
\end{proposition}

\begin{proof}

Fix some $x_0 \in X$. Since $\xi_f \neq \xi_g$, Theorem~\ref{thm:visible} implies that there exists some $r > 0$ such that: if $m,n \geq 0$ and $\gamma: [a,b] \rightarrow X$ is a geodesic segment with $\gamma(a) = f^m(x_0)$ and $\gamma(b) = g^{n}(x_0)$, then there exists some $t \in [a,b]$ such that
\begin{align*}
d(\gamma(t), x_0) \leq r.
\end{align*}
Then by the proof of Corollary~\ref{cor:Gromov_Product}: If $m,n \geq 0$, then 
\begin{align}
d(f^m(x_0), g^{n}(x_0))  \geq d(f^m(x_0),x_0) + d(x_0, g^{n}(x_0)) - 2r \label{eq:dist_est}.
\end{align}

By Theorem~\ref{thm:shadow_lemma} there exists  $R \geq 0$ so that: if $\gamma_1:[a_1,b_1] \rightarrow X$ and $\gamma_2:[a_2,b_2] \rightarrow X$ are $(1,2r)$-quasi-geodesic segments with $\gamma_1(a_1) = \gamma_2(a_2)$ and $\gamma_1(b_1) = \gamma_2(b_2)$, then 
\begin{align*}
\max \left\{ \max_{t \in [a_1,b_1]} d(\gamma_1(t), \gamma_2([a_2,b_2])) , \max_{t \in [a_2,b_2]} d(\gamma_2(t), \gamma_1([a_1,b_1])) \right\} \leq R.
\end{align*} 
Finally, let $C = d(x_0, g(x_0))$. 

Now fix $m \geq 0$. We claim that there exists $n = n(m)>0$ so that
\begin{align*}
d(f^mg^{n}(x_0), x_0) \leq 4r + 2R + C/2.
\end{align*}

Since 
\begin{align*}
\abs{d(g^{n}(x_0), x_0) - d(g^{n+1}(x_0), x_0)} \leq d(g^{n}(x_0), g^{n+1}(x_0)) \leq d(x_0, g(x_0))=C
\end{align*}
for every $n > 0$ and 
\begin{align*}
\lim_{n \rightarrow \infty} d(g^{n}(x_0), x_0) =\infty,
\end{align*}
there exists some $n \geq 0$ so that 
\begin{align*}
\abs{ d(f^m(x_0),x_0) - d(g^{n}(x_0), x_0)} \leq C/2.
\end{align*}

Let $\gamma_1 :[0,T_1] \rightarrow X$ be a geodesic segment with  $\gamma_1(0)= f^mg^{n}(x_0)$ and $\gamma_1(T_1)=f^m(x_0)$. Also let 
$\gamma_2 :[0,T_2] \rightarrow X$ be a geodesic segment with $\gamma_2(0)= f^mg^{n}(x_0)$ and $\gamma_2(T_2)=g^{n}(x_0)$. Finally define the curve $\gamma: [-T_1, T_2] \rightarrow X$ by
\begin{align*}
\gamma(t) = \left\{ 
\begin{array}{ll}
\gamma_1(-t) & \text{ if $t \leq 0$} \\
\gamma_2(t) & \text{ if $t \geq 0$.} 
\end{array}
\right.
\end{align*}

\noindent \textbf{Claim 1.} For all $s, t \in [-T_1, T_2]$ we have 
\begin{align*}
\abs{t-s}-2r \leq d(\gamma(s), \gamma(t)) \leq \abs{t-s}.
\end{align*}
In particular, $\gamma$ is a $(1, 2r)$-quasi-geodesic. 

\begin{proof}[Proof of Claim 1.] Since $\gamma_1$ and $\gamma_2$ are both geodesics, we clearly have 
\begin{align*}
d(\gamma(s), \gamma(t)) \leq \abs{s-t}
\end{align*}
for all $s,t \in [-T_1, T_2]$. Further, if $s$ and $t$ have the same sign, then
\begin{align*}
\abs{t-s} = d(\gamma(s), \gamma(t)).
\end{align*}
Thus it is enough to show that 
\begin{align*}
(t-s) - 2r \leq d(\gamma(s), \gamma(t)).
\end{align*}
for all $-T_1 \leq s \leq 0 \leq t \leq T_2$. In this case we have
\begin{align*}
d(\gamma(s), \gamma(t)) 
&= d(\gamma_1(-s), \gamma_2(t))\\
& \geq d(\gamma_1(T_1), \gamma_2(T_2))-d(\gamma_1(-s), \gamma_1(T_1))-d(\gamma_2(T_2), \gamma_2(t)) \\
& =  d(\gamma_1(T_1), \gamma_2(T_2)) - (T_1+s)-(T_2-t) \\
& = (t-s) + d(f^m(x_0), g^{n}(x_0)) - d(f^m(x_0), f^m g^{n} (x_0)) - d(g^{n}(x_0), f^m g^{n} (x_0)) \\
& \geq (t-s) + d(f^m(x_0), g^{n}(x_0)) - d(x_0, g^{n} (x_0)) - d(x_0, f^m (x_0)).
\end{align*}
So by Equation~\eqref{eq:dist_est} we have
\begin{align*}
d(\gamma(s), \gamma(t))  \geq (t-s) - 2r.
\end{align*}
Hence $\gamma$ is a $(1,2r)$-quasi-geodesic. 
\end{proof}

\noindent \textbf{Claim 2.} $T_2 \leq d(x_0, f^m(x_0)) \leq T_2 + 2r$ and $T_1 \leq d(x_0, g^{n}(x_0)) \leq T_1 + 2r$.

\begin{proof}[Proof of Claim 2.]  Since $f$ and $g$ are 1-Lipschitz we have
\begin{align*}
T_1 = d(f^m(x_0), f^m g^{n}(x_0)) \leq d(x_0, g^{n}(x_0))
\end{align*}
and  
\begin{align*}
T_2 = d(g^{n}(x_0), f^m g^{n}(x_0)) \leq d(x_0, f^m(x_0)). 
\end{align*}
Now 
\begin{align*}
T_1 + T_2 & = d(f^m(x_0), f^m g^{n}(x_0))+d(f^m g^{n}(x_0),g^{n}(x_0)) \\
& \geq d(f^m(x_0), g^{n}(x_0)) \geq d(f^m(x_0),x_0) + d(x_0, g^{n}(x_0)) - 2r
\end{align*}
where we used Equation~\eqref{eq:dist_est} in the last step. Then since $T_1 \leq d(x_0, g^{n}(x_0))$ we then have
\begin{align*}
T_2 + 2r \geq d(f^m(x_0),x_0).
\end{align*}
A similar argument shows that 
\begin{align*}
T_1 + 2r \geq d(x_0,g^{n}(x_0)).
\end{align*}
\end{proof}

Now the previous claim implies that
\begin{align*}
\abs{T_1-T_2} \leq 2r+\abs{d(x_0, f^m(x_0))-d(x_0, g^{n}(x_0))} \leq 2r+C/2 .
\end{align*}

Next let $\sigma : [0,T] \rightarrow X$ be a geodesic segment with $\sigma(0)=f^m(x_0)$ and $\sigma(T) = g^{n}(x_0)$. Then by our choice of $R$, we have
\begin{align*}
d(\sigma(t), \gamma) \leq R
\end{align*}
for all $t \in [0,T]$. So, by the definition of $r$, there exists some $t_0 \in [-T_1,T_2]$ so that 
\begin{align*}
d(\gamma(t_0), x_0) \leq r+R.
\end{align*}

Finally, we show that
\begin{align*}
d(x_0, f^m g^{n}(x_0)) \leq 4r + 2R + C/2.
\end{align*}
First, note that
\begin{align*}
t_0+T_1  & \geq d(\gamma(t_0), \gamma(-T_1)) \geq d(x_0, \gamma(-T_1)) - d(x_0,\gamma(t_0)) \\
&  \geq d(x_0, f^m(x_0)) - r - R \geq T_2 - r - R.
\end{align*}
So 
\begin{align*}
t_0 \geq T_2-T_1-r-R.
\end{align*}
Arguing in a similar way, we also have
\begin{align*}
T_2-t_0 \geq d(\gamma(t_0), \gamma(T_2)) \geq T_1 - r - R
\end{align*}
and
\begin{align*}
t_0 \leq T_2-T_1+r+R.
\end{align*}
So 
\begin{align*}
\abs{t_0} \leq \abs{T_1-T_2}+r+R \leq 3r + R + C/2.
\end{align*}

Thus
\begin{align*}
d(x_0, f^{n} g^m(x_0)) 
&\leq d(x_0, \gamma(t_0)) + d(\gamma(t_0), f^{n} g^m(x_0)) \\
& =  d(x_0, \gamma(t_0))+  d(\gamma(t_0), \gamma(0)) \\
& \leq r+ R + \abs{t_0} \\
& \leq 4r + 2R + C/2.
\end{align*}
\end{proof}

\section{The ends of an unbounded convex domain}\label{Sec:5}

\subsection{The end compactification}\label{sec:end_compactification} In this section we define the end compactification. The end compactification of a topological space was introduced by H. Freudenthal~\cite{Fre31}, see also \cite{P} for a comprehensive survey on the subject. For simplicity we will only consider the case when $X$ is a manifold. Then there exists an increasing sequence $K_0 \subset K_1 \subset K_2 \subset \cdots$ of compact subsets with $X = \cup_{n \geq 0} K_n$. By compactness, each $X \setminus K_n$ has finitely many components. An \emph{end of $X$} is a decreasing sequence $U_0 \supset U_1 \supset U_2 \supset \cdots$ of open sets where each $U_n$ is a connected component of $X\setminus K_n$. Let $E[X]$ denote the set of ends. The set $X \cup E[X]$ has a natural topology making it a compactification of $X$ where each end $(U_j)_{j \geq 0} \in E[X]$ has a neighborhood basis
\begin{align*}
U_k \cup \{ (V_ j)_{j \geq 0} \in E[X] : V_j = U_j \text{ for } j \leq k\}, \quad k \geq 0.
\end{align*}
Then the space $X \cup E[X]$ is compact and Hausdorff. Further, the subset $E[X]$ is closed and totally disconnected. Finally, this compactification does not depend on the choice of compact sets  $K_0 \subset K_1 \subset K_2 \subset \cdots$.

\subsection{Unbounded convex domains} We now make some observations about the structure of unbounded convex domains. 

For $x\in \R^l$ and $R>0$ we let 
\[
B(x,R):=\{w\in \R^l: \|w-z\|<R\}
\]
 be the Euclidean ball of center $x$ and radius $R$.

\begin{definition}
Let $D\subset \R^l$ be an unbounded convex domain. A vector $v\in \R^l$, $\|v\|=1$, is called a {\sl direction at $\infty$ for $D$} if there exists $x\in D$ such that $x+tv\in D$ for all $t\geq 0$. Then let $S_\infty(D) \subset \R^l$ be the set of directions at infinity for $D$. 
\end{definition}

By convexity of $D$, if $v$ is a direction at $\infty$ for $D$, then for every $z\in D$ and all $t\geq 0$ it holds $z+tv\in D$.

\begin{lemma}\label{two-components}
Let $D\subset \R^l$ be an unbounded convex domain. Then there exists at least one direction $v$ at $\infty$ for $D$.  Moreover 
\begin{enumerate}
\item either $D\setminus\overline{B(0,R)}$ has only one unbounded connected component for all $R>0$ or
\item there exists $R_0>0$ such that $D\setminus\overline{B(0,R)}$ has two unbounded connected components for all $R\geq R_0$. This is the case if and only if the only directions at $\infty$ for $D$ are $v$ and $-v$. 
\end{enumerate}
\end{lemma}

\begin{proof} This follows immediately from convexity, see Lemma 2.2 in~\cite{BrGa17} for details (in~\cite{BrGa17} the result was stated for unbounded convex domains in $\C^d$, but the proof is valid without any modification when replacing $\C^d$ with $\R^l$). \end{proof}

\begin{lemma}\label{product-two-components} Let $D \subset \R^l$ be an unbounded convex domain and $S_\infty(D) = \{ v, -v\}$ for some $v \in \mathbb{R}^l$. Let $H$ be the real orthogonal complement of $\mathbb R v$ in $\R^l$. Then there exists a bounded convex domain $\Omega \subset H$ such that $D = \Omega +\mathbb R v$.
\end{lemma}

\begin{proof}
 Let $\Omega:=D\cap H$. The set $\Omega$ is an open convex set in $H$, and, since every direction at $\infty$ for $\Omega$ is also a direction at $\infty$ for $H$, the set $\Omega$ must be  bounded. Take $p\in D$. Since $p+tv\in D$ for all $t\in \R$, then there exists $t_0\in \R$ such that $p':=p-t_0v\in H$. Hence, $p=p'+t_0v$. Since $p \in D$ was arbitrary, $D = \Omega +\mathbb R v$.
\end{proof}

From the last two lemmas we have the following Corollary. 

\begin{corollary}
Let $D\subset \R^l$ be an unbounded convex domain. Then $\overline{D}$ has either one or two ends.  Moreover, 
\begin{enumerate}
\item $\overline{D}$ has one end if and only if for every $R>0$ the open set $D\setminus\overline{B(0,R)}$ has only one unbounded connected component,
\item $\overline{D}$ has two ends if and only if $S_\infty(D) = \{ v, -v\}$ for some $v \in \mathbb{R}^l$.
\end{enumerate}
\end{corollary}

\subsection{The Gromov boundary and ends} 

Assume now that $D\subset \C^d$ is an unbounded $\C$-proper convex domain such that  $(D,K_D)$ is Gromov hyperbolic. Throughout the section we will let $\overline{D} \subset \C^d$ denote the closure of $D$ in $\C^d$ and $\partial D = \overline{D} \setminus D$.  

To distinguish between the Gromov compactification and the End compactification, we will write $\xi_n \overset{Gromov}{\longrightarrow} \xi$ when $\xi_n \in \overline{D}^G$ is a sequence converging to $\xi \in \overline{D}^G$. 

The main result of this section is the following. 

\begin{proposition}\label{prop:end_convergence} Suppose $x$ is an end of $\overline{D}$. Then there exists $\zeta_x \in \partial_G D$ such that: if $z_n \in D$ converges to $x$ in $\overline{D}^{\star}$, then 
\begin{align*}
z_n \overset{Gromov}{\longrightarrow} \zeta_x.
\end{align*}
Moreover, if $\overline{D}$ has two ends $x,y$, then $\zeta_x \neq \zeta_y$. 
\end{proposition}

The proof of Proposition~\ref{prop:end_convergence} will require several lemmas.

\begin{lemma}\label{lem:conG1} For any $v \in S_\infty(D)$, there exists a point $\zeta_v \in \partial_G D$ such that if $p_n \in D$ is a sequence with $\norm{p_n} \rightarrow \infty$ and $p_n / \norm{p_n} \rightarrow v$, then $p_n \overset{Gromov}{\longrightarrow} \zeta_v$. 
\end{lemma}

\begin{proof} First consider the map 
\begin{align*}
z \in D \mapsto z+v \in D.
\end{align*}
Then by Karlsson's Theorem~\ref{thm:gromov_wd} there exists some $\zeta_v \in \partial_G D$ such that: 
\begin{align*}
z+nv \overset{Gromov}{\longrightarrow} \zeta_v
\end{align*}
for all $z \in D$. 

Now fix a sequence $p_n \in D$ with $\norm{p_n} \rightarrow \infty$ and $p_n / \norm{p_n} \rightarrow v$. Assume for a contradiction that $p_n$ does not converge to $\zeta_v$ in $\overline{D}^G$. Then by passing to a subsequence we can suppose that $p_{n} \overset{Gromov}{\longrightarrow}\xi \in \partial_G D$, with $\xi \neq \zeta_v$.

\vspace{1mm}
Consider, for $n \geq 0$,  the function $b_{n} : D \rightarrow \mathbb{R}$ defined by 
\begin{align*}
b_{n}(z) = K_D(z,p_{n}) - K_D(p_{n},z_0).
\end{align*}
Since each $b_{n}$ is $1$-Lipschitz with respect to the Kobayashi metric and $b_{n}(z_0)=0$, by passing to a subsequence we can suppose that $b_{n}$ converges uniformly on compacta to some function $b$.

\vspace{1mm}
\noindent{\bf Claim.} For each $n$, the set $b_{n}^{-1}( (-\infty,0])$ is convex.

\begin{proof}[Proof of Claim]
By Proposition 3.2 in~\cite{BS2009} for every $z \in D$ and every $r > 0$, the closed metric ball 
\begin{align*}
\overline{B^K_{D}(z,r)}:=\{w \in D:\ K_D(z,w) \leq r\},
\end{align*}
with center $z$ and radius $r$ is a closed convex subset of $D$. In particular, the set $b_{n}^{-1}( (-\infty,0])=\overline{B^K_{D}(p_{n},K_D(p_{n},z_0))}$ is convex for every $n \geq 0$.
\end{proof}

\vspace{1mm}
Consequently, for every $n \geq 0$, the set $b_{n}^{-1}( (-\infty,0])$ contains the line segment 
\begin{align*}
[z_0, p_{n}]=\{ tz_0 + (1-t)p_n : 0 \leq t \leq 1\}.
\end{align*}
 Since $\lim_{n \rightarrow \infty}(p_{n} /\norm{p_{n}}) = v$, then the set $b^{-1}((-\infty, 0])$ contains the real line 
\begin{align*}
z_0 + \mathbb{R}_{\geq 0} \cdot v.
\end{align*}
Let, for every $m \geq 0$, $z_{m}:=z_0 + m v$. Since $\xi \neq \zeta_v$, by Corollary~\ref{cor:Gromov_Product} there exists some $M >0$ such that (after extracting subsequences if necessary): 
\begin{align*}
K_\Omega(z_m,p_{n}) \geq K_\Omega(z_m,z_0) +K_\Omega(z_0,p_{n}) -M
\end{align*}
for every $n \geq 0$ and $m \geq 0$.
Then we have, for every $m \geq 0$:
\begin{align*}
0 \geq b(z_m) = \lim_{n \rightarrow \infty} K_\Omega(z_m,p_{n}) - K_\Omega(p_{n},z_0) \geq K_\Omega(z_m,z_0)-M.
\end{align*}
Since $K_\Omega(z_m,z_0) \rightarrow \infty$, we obtain a contradiction. 
\end{proof}

\begin{lemma}\label{conv2-lem} Suppose that $\overline{D}$ has one end. Then $\zeta_v = \zeta_{w}$ for all $v,w \in S_\infty(D)$. 
\end{lemma}

\begin{proof} We first consider the case where $v,w \in S_\infty(D)$ are linearly independent over $\mathbb{R}$. Suppose for a contradiction that $\zeta_v \neq \zeta_w$. Consider the maps $f,g : D \rightarrow D$ defined by 
\begin{align*}
f(z) = z+v \text{ and } g(z) = z+w.
\end{align*}
Then $f^m(z) \overset{Gromov}{\longrightarrow} \zeta_v$ and $g^n(z) \overset{Gromov}{\longrightarrow} \zeta_w$ for all $z \in D$ by Lemma~\ref{lem:conG1}. So by Proposition~\ref{prop:commuting} there exist $m_k, n_k \rightarrow \infty$ such that 
\begin{align*}
\lim_{k \rightarrow \infty} f^{m_k}g^{n_k}(z_1) = z_2
\end{align*}
for some $z_1,z_2 \in D$. But 
\begin{align*}
f^{m_k}g^{n_k}(z_1) = z_1+m_kv+n_kw
\end{align*}
and $v,w$ are linearly independent. Therefore, this is impossible and $\zeta_v = \zeta_w$. 

Now if $v,w \in S_\infty(D)$ are linearly dependent over $\mathbb{R}$ and distinct, then $w=-v$. Since $\overline{D}$ has one end, there exists some $u \in S_\infty(D)$ such that $u,v$ are linearly independent over $\mathbb{R}$. Then $\zeta_v = \zeta_u = \zeta_w$. 
\end{proof}

\begin{lemma}\label{conv-lem} Suppose that $\overline{D}$ has two ends, that is $S_\infty(D) = \{ v, -v\}$ for some $v \in \mathbb{C}^d$. Then $\zeta_v \neq \zeta_{-v}$. 
\end{lemma}

\begin{proof} We start by establishing the following claim: 

\vspace{1mm}
\noindent{\bf Claim.} Let $z_0 \in D$. Then there exist $A>1$ such that the curve $\sigma : \mathbb{R} \rightarrow D$ given by $\sigma(t) = z_0 + tv$ is an $(A,0)$-quasi-geodesic.

\begin{proof}[Proof of the claim.] For $z \in D$ and $w \in \mathbb{C}^d$, let 
\begin{align*}
\delta_D(z;w): =\inf\{ \norm{z-u} : u \in \partial D \cap (z + \mathbb{C} \cdot w) \}.
\end{align*}
The usual estimates on the Kobayashi metric of convex domains (see, {\sl e.g.}, \cite{BP}) give
\begin{align*}
\frac{ \norm{w}}{2\delta_D(z;w)} \leq k_D(z;w) \leq \frac{ \norm{w}}{\delta_D(z;w)}
\end{align*}
for all $z \in D$ and $w \in \mathbb{C}^d$.

Since $D + tv = D$ for all $t \in \mathbb{R}$, we see that 
\begin{align*}
\delta_D(z;w) = \delta_D(z+tv;w)
\end{align*}
for all $z \in D$, $w \in \mathbb{C}^d$ and $t \in \mathbb{R}$. This implies that 
\begin{align*}
\delta_D(\sigma(t); \sigma^\prime(t)) = \delta_D(z_0+tv; v) = \delta_D(z_0;v)
\end{align*} 
for all $t \in \mathbb{R}$. Further, by Lemma \ref{product-two-components} there exists $\alpha > 0$ such that 
\begin{align*}
\delta_D(z;w)  \leq \alpha
\end{align*}
for all $z \in D$ and $w \in \mathbb{C}^d$.

Now fix $a \leq b$. Then, 
\begin{align*}
K_D(\sigma(a), \sigma(b)) \leq\int_a^b k_D(\sigma(t);\sigma'(t))dt \leq \int_a^b \frac{ \norm{\sigma^\prime(t)}}{\delta_D(\sigma(t); \sigma^\prime(t))} dt = \frac{1}{\delta_D(z_0;v)}(b-a).
\end{align*}
Now take any $C^1$-curve $\gamma : [0,1] \rightarrow D$ such that $\gamma(0) = \sigma(a) = z_0+av$ and $\gamma(1) = \sigma(b) = z_0 + bv$.  Then 
\begin{align*}
\ell_D(\gamma) \geq \int_0^1 \frac{ \norm{ \gamma^\prime(t)}}{2\delta_D(\gamma(t); \gamma^\prime(t))} dt \geq \frac{1}{2\alpha} \int_0^1 \norm{ \gamma^\prime(t)} dt \geq \frac{1}{2\alpha} \norm{ \gamma(1)-\gamma(0)} = \frac{1}{2\alpha}(b-a).
\end{align*}
Since $\gamma$ was an arbitrary $C^1$ curve joining $\sigma(a)$ to $\sigma(b)$ we see that 
\begin{align*}
K_D(\sigma(a),\sigma(b)) \geq \frac{1}{2\alpha}(b-a).
\end{align*}
The previous estimates show that $\sigma$ is a $(A,0)$-quasi-geodesic for some $A>1$. 
\end{proof} 

Then Remark~\ref{obs:limits_of_quasi_geod} implies that the (Gromov) limits 
\begin{align*}
\lim_{t \rightarrow \infty} \sigma(t) \text{ and } \lim_{t \rightarrow -\infty} \sigma(t)
\end{align*}
exist in $\partial_G D$ and are distinct. So $\zeta_v\neq \zeta_{-v}$.
\end{proof}

\begin{proof}[Proof of Proposition~\ref{prop:end_convergence}]
Follows immediately from Lemmas~\ref{lem:conG1},~\ref{conv2-lem}, and~\ref{conv-lem}. 
\end{proof}

\section{Proof of Theorem~\ref{ext-thm}.}\label{Sec:6}

We begin the proof of Theorem~\ref{ext-thm} by establishing some basic properties of geodesics in $(D,K_D)$ when $D$ is convex and $(D,K_D)$ is Gromov hyperbolic. 

Throughout the section we will let $\overline{D} \subset \C^d$ denote the closure of $D$ in $\C^d$ and $\partial D = \overline{D} \setminus D$. Also, for $z,w \in \C^d$ let 
\begin{align*}
[z,w] = \{ tz + (1-t)w : 0 \leq t \leq 1\}
\end{align*}
denote the Euclidean line segment joining them. 

\begin{lemma}\label{infty-lem}
Let $D \subset \C^d$ be an unbounded $\C$-proper convex  domain. If $z_{n}, w_{n} \subset D$ are sequences with $\lim_{n \rightarrow \infty}\|z_{n}\| = \infty$ and $\lim_{n \rightarrow \infty}w_{n} =\xi \in \partial D$, then $\lim_{n \rightarrow \infty}K_{D}(z_{n},w_{n}) = \infty$.
\end{lemma}

\begin{proof}
 According to Theorem 7.6 in~\cite{Fr91} there is a complex affine isomorphism $A:\C^d\rightarrow \C^d$ such that $A(D) \subset \mathbb H^d$, where $\mathbb H:=\{z \in \C:\ \Re(z) < 0\}$. Then:
$$
K_D(z_{n},w_{n}) = K_{A(D)}(Az_n,Aw_n) \geq K_{\mathbb H^d}(Az_n,Aw_n)\rightarrow \infty
$$
as $n \rightarrow \infty$.
\end{proof}

\begin{lemma}\label{lem:bd_dist} Let $D$ be a $\C$-proper convex domain in $\C^d$ and suppose that $(D,K_D)$ is  Gromov hyperbolic. If $z_n,w_n \in D$ are sequences with $\lim_{n \rightarrow \infty} z_n = \xi \in \partial D$ and 
\begin{align*}
\sup K_D(z_n, w_n) < + \infty,
\end{align*}
then $w_n \rightarrow \xi$. 
\end{lemma}

\begin{proof}
Since $\overline{D}^{\star}$ is compact we can assume that $w_n \rightarrow \eta$ for some $\eta \in \overline{D}^{\star}$. By Lemma~\ref{infty-lem} we must have $\eta \in \partial D$. Suppose for a contradiction that $\xi \neq \eta$. Since every convex domain is also $\mathbb{C}$-convex and 
\begin{align*}
\sup K_D(z_n, w_n) < + \infty,
\end{align*}
Proposition 3.5 in~\cite{Zi17} implies that $\partial D$ contains a complex affine disk in its boundary. However, since $(D,K_D)$ is Gromov hyperbolic, \cite[Theorem 3.1]{Zi16} says that $\partial D$ does not contain any non trivial analytic disc. This is a contradiction. 
\end{proof}

\begin{lemma} Let $D$ be a $\C$-proper convex domain in $\C^d$ and suppose that $(D,K_D)$ is  Gromov hyperbolic. If $\sigma: [0,+\infty) \rightarrow D$ is a geodesic ray, then 
$
\lim_{t \rightarrow \infty} \sigma(t)
$
exists in $\overline{D}^{\star}$. 
\end{lemma}

\begin{proof} 
Let $L \subset \overline{D}^{\star}$ denote the set of points $x \in\overline{D}^{\star}$ where there exists $t_n \rightarrow \infty$ such that $\sigma(t_n) \rightarrow x$. Suppose for a contradiction that $L$ is not a single point. 

Notice that $L$ is connected and so $L$ contains at least one point in $\partial D$ (since the space $\overline{D}^{\star}$ is Hausdorff). Then, again by connectedness, $L$ must contain at least two points in $\partial D$. So we can find $a_n, b_n \rightarrow \infty$ and distinct $\xi, \eta \in \partial D$ such that $\sigma(a_n) \rightarrow \xi$ and $\sigma(b_n) \rightarrow \eta$. We may also assume that $a_n \leq b_n$ for all $n \in \mathbb{N}$. 

Also by the definition of the Gromov boundary, if $t_n \rightarrow \infty$, then 
\begin{align*}
\sigma(t_n) \overset{Gromov}{\longrightarrow} [\sigma].
\end{align*}

Now fix some $z_0 \in D$. By~\cite[Lemma 3.2]{Zi16} there exists some $A \geq 1$ such that the line segments $[z_0, \sigma(b_n)]$ are $(A,0)$-quasi-geodesics. Then by Theorem~\ref{thm:shadow_lemma}, there exists some $R > 0$ and $z_n \in [z_0, \sigma(b_n)]$ such that 
\begin{align*}
K_D(z_n, \sigma(a_n)) \leq R
\end{align*}
for all $n$. Since $\sigma(b_n) \rightarrow \eta$,  $z_n \in [z_0, \sigma(b_n)]$, and
\begin{align*}
\lim_{n \rightarrow \infty} K_D(z_n, z_0) \geq \lim_{n \rightarrow \infty} K_D(\sigma(a_n), \sigma(0)) - K_D(\sigma(0),z_0) - R = \infty
\end{align*}
we see that $z_n \rightarrow \eta$. Since $\eta \neq \xi$, this contradicts Lemma~\ref{lem:bd_dist}.
\end{proof}

\begin{lemma}\label{lem:cont} Let $D$ be a $\C$-proper convex domain in $\C^d$ and suppose that $(D,K_D)$ is  Gromov hyperbolic. If $T_n \in (0,+\infty]$, $\sigma_n: [0,T_n) \rightarrow D$ is a sequence of geodesics, and $\sigma_n$ converges locally uniformly to a geodesic $\sigma: [0,+\infty) \rightarrow D$, then 
\begin{align*}
\lim_{t \rightarrow \infty} \sigma(t) = \lim_{n \rightarrow \infty} \lim_{t \rightarrow T_n} \sigma_n(t).
\end{align*}
\end{lemma}

\begin{proof} The proof is nearly identical to the proof of the previous Lemma. Since $\overline{D}^{\star}$ is compact, it is enough to consider the case when 
\begin{align*}
\lim_{n \rightarrow \infty} \lim_{t \rightarrow T_n} \sigma_n(t)
\end{align*}
exists in $\overline{D}^\star$. Let $\xi = \lim_{t \rightarrow \infty} \sigma(t) \in \overline{D}^{\star}$, $\xi_n = \lim_{t \rightarrow T_n} \sigma_n(t)\in \overline{D}^{\star}$, and 
\begin{align*}
\xi_\infty = \lim_{n \rightarrow \infty} \lim_{t \rightarrow T_n} \sigma_n(t) \in \overline{D}^{\star}.
\end{align*}
Suppose for a contradiction that $\xi \neq \xi_\infty$. 

Since $\lim_{n \rightarrow \infty} \sigma_n(t) = \sigma(t)$ for every $t$, we can pick $a_n \rightarrow \infty$ such that $\sigma_n(a_n) \rightarrow \xi$. We can also pick $b_n \rightarrow \infty$ such that $\sigma_n(b_n) \rightarrow \xi_\infty$. 

\vspace{1mm}
\noindent{\bf Claim.} After possibly passing to a subsequence, there exists $a_n \leq c_n \leq b_n$ such that $\sigma_n(c_n)$ converges to $\eta \in \partial D$ and $\eta \neq \xi$. 

\begin{proof}[Proof of Claim.] Define a distance $d$ on $\overline{D}^{\star}$ as follows: 
\begin{align*}
d(x,y) = \left\{ \begin{array}{ll} 
0 & \text{ if } x=y \\
\norm{x-y} & \text{ if } x,y \in \C^d \\
\infty & \text{ if } x \neq y \text{ and at least one of } x,y \text{ is an end.} \\
\end{array}\right.
\end{align*}
Since $\xi\neq \xi_\infty$, we can pick $a_n \leq c_n \leq b_n$ such that 
\begin{align*}
& \liminf_{n \rightarrow \infty}  \ d(\sigma_n(a_n), \sigma_n(c_n)) > 0, \\
& \liminf_{n \rightarrow \infty}  \ d(\sigma_n(b_n), \sigma_n(c_n)) > 0, \text{ and }\\
& \limsup_{n \rightarrow \infty} \  \norm{c_n} < \infty.
\end{align*}
Such a sequence clearly exists when at least one of $\xi$, $\xi_\infty$ is not an end and Lemma~\ref{product-two-components} implies the existence of the sequence $c_n$ when both $\xi$, $\xi_\infty$ are ends.

Then we can pass to a subsequence such that $\sigma_n(c_n)$ converges to $\eta \in \partial D$ and $\eta \neq \xi$. 
\end{proof}

Now fix some $z_0 \in D$. By~\cite[Lemma 3.2]{Zi16} there exists some $A > 1$ such that the line segments $[z_0, \sigma_n(b_n)]$ are $(A,0)$-quasi-geodesics. Then by Theorem~\ref{thm:shadow_lemma}, there exists some $R > 0$ and $z_n \in [z_0, \sigma_n(b_n)]$ such that 
\begin{align*}
K_D(z_n, \sigma_n(a_n)) \leq R
\end{align*}
for all $n$. Since $\sigma_n(b_n) \rightarrow \eta$,  $z_n \in [z_0, \sigma_n(b_n)]$, and
\begin{align*}
\lim_{n \rightarrow \infty} K_D(z_n, z_0) \geq \lim_{n \rightarrow \infty} K_D(\sigma_n(a_n), \sigma_n(0)) - K_D(\sigma_n(0),z_0) - R = \infty
\end{align*}
we see that $z_n \rightarrow \eta$. Since $\eta \neq \xi$, this contradicts Lemma~\ref{lem:bd_dist}.
\end{proof}

\begin{lemma}\label{lem:well_defined} Let $D$ be a $\C$-proper convex domain in $\C^d$ and suppose that $(D,K_D)$ is Gromov hyperbolic. If $\sigma_1, \sigma_2 : [0,+\infty) \rightarrow D$ are geodesics, then  
\begin{align*}
\lim_{t \rightarrow \infty} \sigma_1(t) = \lim_{t \rightarrow \infty} \sigma_2(t) 
\end{align*}
in $\overline{D}^{\star}$ if and only if $[\sigma_1]=[\sigma_2]$. 
\end{lemma}

\begin{proof} First suppose that $[\sigma_1]=[\sigma_2]$ and let $\xi_j= \lim_{t \rightarrow \infty} \sigma_j(t)$ in $\overline{D}^{\star}$. Since 
\begin{align*}
\sup_{t \geq 0} K_D(\sigma_1(t), \sigma_2(t)) < +\infty, 
\end{align*}
Lemma~\ref{infty-lem} implies that $\xi_1 \in \mathbb C^d$ if and only if $\xi_2 \in \mathbb C^d$. If $\xi_1 \notin \mathbb C^d$, then $\xi_1 = \xi_2$ by  the ``moreover'' part of Proposition~\ref{prop:end_convergence}. If $\xi_1 \in \mathbb C^d$, then Lemma~\ref{lem:bd_dist} implies that $\xi_1 = \xi_2$. So in either case
\begin{align*}
\lim_{t \rightarrow \infty} \sigma_1(t) = \lim_{t \rightarrow \infty} \sigma_2(t).
\end{align*}

Next suppose that 
\begin{align*}
\lim_{t \rightarrow \infty} \sigma_1(t) = \lim_{t \rightarrow \infty} \sigma_2(t) = \xi \in \overline{D}^{\star}.
\end{align*}
 If $\xi \notin \mathbb C^d$, then $[\sigma_1]=[\sigma_2]$ by Proposition~\ref{prop:end_convergence}. So we may assume that $\xi \in \mathbb C^d$. 

Fix $T > 0$. We will bound $K_D(\sigma_1(T), \sigma_2(T))$ from above.  Then fix some $z_0 \in D$. By~\cite[Lemma 3.2]{Zi16} there exists some $A \geq 1$ such that the line segments $[z_0, \sigma_j(t)]$ are $(A,0)$-quasi-geodesics. Then by Theorem~\ref{thm:shadow_lemma}, there exists some $R > 0$ such that: for every $t \geq T$, there exists  $z_t \in [z_0, \sigma_1(t)]$ with
\begin{align*}
K_D(z_t, \sigma_1(T)) \leq R.
\end{align*}
Since 
\begin{align*}
\lim_{t \rightarrow \infty} \sigma_1(t) = \lim_{t \rightarrow \infty} \sigma_2(t) 
\end{align*}
and $K_D(z_t, \sigma_1(0)) \leq T + R$, there exists $w_t \in  [z_0, \sigma_2(t)]$ such that 
\begin{align*}
\lim_{t \rightarrow \infty} K_D(w_t, z_t) =0.
\end{align*}
Now fix $t $ sufficiently large such that 
\begin{align*}
K_D(w_t, z_t) \leq 1.
\end{align*}
By Theorem~\ref{thm:shadow_lemma} there exists $S \in [0,t]$ such that 
\begin{align*}
K_D(\sigma_2(S), w_t) \leq R.
\end{align*}
Then 
\begin{align*}
K_D(\sigma_1(T), \sigma_2(S)) \leq 2R+1.
\end{align*}
Since 
\begin{align*}
2R+1 & \geq K_D(\sigma_1(T), \sigma_2(S)) \geq \abs{ K_D(\sigma_1(T), \sigma_1(0)) - K_D(\sigma_2(S), \sigma_2(0))}-K_D(\sigma_1(0), \sigma_2(0)) \\
& = \abs{T-S} - K_D(\sigma_1(0), \sigma_2(0)),
\end{align*}
we see that 
\begin{align*}
K_D(\sigma_1(T), \sigma_2(T)) \leq K_D(\sigma_1(T), \sigma_2(S)) + \abs{T-S} \leq 4R+2 + K_D(\sigma_1(0), \sigma_2(0)).
\end{align*}

Since $T > 0$ was arbitrary, we have 
\begin{align*}
\sup_{t \geq 0} K_D(\sigma_1(t), \sigma_2(t)) < +\infty
\end{align*}
and hence $[\sigma_1]=[\sigma_2]$. 
\end{proof}

\subsection{Proof of Theorem~\ref{ext-thm}}

Define $\Phi: \overline{D}^G \rightarrow \overline{D}^{\star}$ by $\Phi(z) = z$ when $z \in D$ and 
\begin{align*}
\Phi([\sigma]) = \lim_{t \rightarrow \infty} \sigma(t)
\end{align*}
when $[\sigma] \in \partial_G D$. By Lemma~\ref{lem:well_defined}, $\Phi$ is well defined and one-to-one. By Lemma~\ref{lem:cont}, $\Phi$ is continuous. Since $\Phi(D) = D$, $\overline{D}^G$ is compact, and $D$ is dense in $\overline{D}^{\star}$, it follows that $\Phi$ is onto. The map $\Phi$ being continuous, one-to-one and onto, between compact Hausdorff spaces, it is a homeomorphism.

\section{Proof of Corollary~\ref{eucl-thm},  Theorem~\ref{commut-thm} and Corollary~\ref{str-cor}}\label{Sec:7}

\begin{proof}[Proof of Corollary~\ref{eucl-thm}] 

Since $F$ is a homeomorphism, it must be a proper map. Thus, since $(D,K_D)$ is a proper metric space, we see that $(\Omega, K_\Omega)$ is a proper metric space. So $\Omega$ is $\mathbb{C}$-proper by Theorem~\ref{thm:barth}. 

According to \cite{BaBo00}, if $D$ is strongly  pseudoconvex, then $(D,K_D)$ is Gromov hyperbolic and the identity map $\id_D: D \rightarrow D$ extends to a homeomorphism ${\widetilde\id}_D:\overline{D}
\rightarrow\overline{D}^G$ (and $\overline{D}=\overline{D}^\star$). On the other hand, if $D$ is convex, then Theorem~\ref{ext-thm} implies that the identity map $\id_D: D \rightarrow D$ extends to a homeomorphism ${\widetilde\id}_D:\overline{D}^{\star}
\rightarrow\overline{D}^G$.

Since $F:(D,K_D) \rightarrow (\Omega, K_\Omega)$ is a quasi-isometry, $(\Omega,K_{\Omega})$ is also Gromov hyperbolic (see for instance \cite{GhHa}, Theorem~12, p.88). Then Theorem~\ref{ext-thm} implies that the identity map $\id_\Omega: \Omega \rightarrow \Omega$ extends to a homeomorphism ${\widetilde\id}_\Omega:\overline{\Omega}^\star
\rightarrow\overline{\Omega}^G$.

Finally, since $F: (D,K_D) \rightarrow (\Omega,K_{\Omega})$ is a quasi-isometry and $F$ is a homeomorphism,  $F$ extends to a homeomorphism $\widetilde{F}:\overline{D}^G \rightarrow \overline{\Omega}^G$ (see for instance \cite{GhHa}, Proposition~14, p.128).  Hence, 
$F$ extends to the homeomorphism $({\widetilde\id}_\Omega)^{-1}\circ \widetilde{F}\circ {\widetilde\id}_D:\overline{D}^{\star} \rightarrow \overline{\Omega}^{\star}$. 
\end{proof}

\begin{proof}[Proof of Theorem~\ref{commut-thm}] Let $(X,d)$ be a proper geodesic Gromov hyperbolic metric space and let $f, g$ be 1-Lipschitz self-maps which commute under composition and satisfy \eqref{com-eq} (or equivalently \eqref{different Wolff}). Then the family $\{f^m \circ g^{n}\}_{m,n \in \mathbb N}$ is equicontinuous and according to Proposition~\ref{prop:commuting}, for every $m \geq 0$, there exists $n(m) \geq 0$ such that for every $x \in X$, the set $\{f^m \circ g^{n(m)}(x)\}$ is relatively compact in $X$. Then it follows from the Ascoli-Arzel\`a Theorem that there exist sequences $\{m_k\}, \{n_k\}\subset \mathbb N$ such that $f^{m_k}\circ g^{n_k}$ converges uniformly on compacta to a 1-Lipschitz map $h:X \to X$. Moreover, we may assume
\[
p_k:=m_{k+1}-m_k\to \infty, \quad p'_k:=n_{k+1}-n_k\to \infty,
\]
and
\[
q_k:=p_k-m_k\to \infty, \quad q^\prime_k:=p'_k-n_k\to \infty.
\]
Since
\[
(f^{p_k}\circ g^{p'_k})(f^{m_k}(g^{n_k}(z)))=(f^{m_{k+1}}\circ g^{n_{k+1}})(z)\to h(z),
\]
it follows that, up to subsequences, $f^{p_k}\circ g^{p'_k}$ converges uniformly on compacta to a 1-Lipschitz self-map $\rho:X\to X$, and, moreover
\[
h\circ \rho=\rho\circ h=h. 
\]
Next,
\[
(f^{q_k}\circ g^{q'_k})(f^{m_k}(g^{n_k}(z)))=(f^{p_k}\circ g^{p_k'})(z)\to \rho(z),
\] 
hence, again passing to a subsequence if necessary, $f^{q_k}\circ g^{q'_k}$ converges uniformly on compacta to a 1-Lipschitz map $\chi:X \to X$ such that
\[
\chi \circ h=h\circ \chi =\rho.
\]
Therefore,
\[
\rho \circ \rho=\chi \circ h\circ \rho=\chi \circ h=\rho.
\]
The statement follows then from the identities $f\circ \rho=\rho\circ f$ and $g\circ \rho=\rho\circ g$, which imply in particular that $f(M)\subseteq M$, $g(M)\subseteq M$. 

Now, since
\[
(f^{p_k-1}\circ g^{p'_k-1})\circ (f\circ g)\to \rho,
\]
passing to a subsequence if necessary, $f^{p_k-1}\circ g^{p'_k-1}$ converges uniformly on compacta to a 1-Lipschitz map $\psi:X \to X$. Moreover, $\psi(M)\subset M$. Hence, for $z\in M$, 
\[
(\psi \circ (f\circ g))(z)=z.
\]
Therefore, $f\circ g$ is an automorphism of $M$. Since $f\circ g=g\circ f$, it follows that both $f$ and $g$ are automorphisms of $M$. This proves Theorem~\ref{commut-thm}. 
\end{proof}

\begin{proof}[Proof of Corollary~\ref{str-cor}.] Since holomorphic maps are 1-Lipschitz for the Kobayashi distance and since the Euclidean boundary of a $C^2$-smooth strongly pseudoconvex domain can be identified with its Gromov boundary (see~\cite{BaBo00}), there exist $M$ and  $\rho$ as in Theorem~\ref{commut-thm}. Moreover, $\rho$ is holomorphic as a limit, on compacta, of holomorphic maps, and $M$ is a holomorphic retract of $D$.

We first show  that $M$ is biholomorphic to $\mathbb B^k$, for some $1\leq k\leq d$. Let $z_0 \in M$ and let $i:M \rightarrow D$ be the inclusion map. Then the sequence $\{z_{n}:=f^{n} \circ i(z_0)\}$ converges to $p_f \in \partial D$. Since $D$ is strongly pseudoconvex, there exist, for every $n$, $0 < r_{n} < 1$ with $\lim_{n \rightarrow \infty} r_{n} = 1$ and a holomorphic injective map $\sigma_{n}: D \rightarrow \mathbb B^n$ such that 
\begin{align*}
\{ z \in \mathbb{C}^d : \norm{z} < r_n\} \subset \sigma_{n}(D)
\end{align*}
 and $\sigma_{n}(z_{n}) = 0$  (see~\cite{DGZ2016, DFW2014}).

We denote by $\varphi_{n}:=\sigma_{n} \circ f^{n} \circ i$ and $\Omega_{n}:=\sigma_{n}(D)$. Then $\varphi_{n}$ is a holomorphic isometry from $(M,K_M)$ into $(\Omega_{n},K_{\Omega_{n}})$. Since $\Omega_{n}$ converges, for the local Hausdorff convergence, to $\mathbb B^n$, we may extract from $\{\varphi_{n}\}$ a subsequence, still denoted $\{\varphi_{n}\}$, that converges uniformly on compact subsets of $M$ to some holomorphic isometry $\psi:M \rightarrow \mathbb B^d$. It follows that $\psi(M)$ is a complex totally geodesic submanifold of $\mathbb B^n$ passing through the origin. Hence after post composing $\psi$ with a rotation, we can assume that $\psi(M)=\mathbb B^k \times \{0\}$, where $k$ is the complex dimension of $M$.

Now view $\psi$ as a biholomorphism from $M$ to $\B^k$. Then $\tilde f:=\psi\circ f \circ \psi^{-1}$ and $\tilde g:=\psi\circ g \circ \psi^{-1}$ are commuting automorphisms of $\B^k$ with no fixed points in $\B^k$. Let $x_{\tilde f}, x_{\tilde g}\in \partial \B^k$ be the Denjoy-Wolff point of $\tilde f$, $\tilde g$. If $x_{\tilde f}=x_{\tilde g}$, then $\{\tilde g^m \circ \tilde f^n(z)\}_{m,n\in \N}$ is compactly divergent for every $z\in \B^k$, contradicting   Proposition \ref{prop:commuting}. Hence, $x_{\tilde f}\neq x_{\tilde g}$. It follows  (see \cite{Br99}) that $\tilde f, \tilde g$ are hyperbolic automorphisms of $\B^k$. Moreover, let  $\tilde \Delta\subset \B^k$ be  the unique complex geodesic such that $x_{\tilde f}, x_{\tilde g}\in \overline{\tilde \Delta}$ ($\tilde \Delta\subset \B^k$ is, in fact, the intersection of $\B^k$ with the affine complex line joining $x_{\tilde f}$ with $x_{\tilde g}$). Then (see \cite{Br99}), $\tilde f(\tilde \Delta)=\tilde \Delta$ and $\tilde g(\tilde \Delta)=\tilde \Delta$. It follows that $\Delta:=\psi^{-1}(\tilde\Delta)$ is a complex geodesic in $M$, hence in $D$, such that $f(\Delta)=\Delta$ and $g(\Delta)=\Delta$. Clearly, $p_f, p_g\in \partial \Delta$. Moreover, since $\tilde \Delta$ is a holomorphic retract of $\B^k$, then $\Delta$ is a holomorphic retract of $M$. Hence, since $M$ is a holomorphic retract of $D$, it follows that $\Delta$ is a holomorphic retract of $D$.

The last statement about the existence of admissible limits for $f$ at $p_g$ and $g$ at $p_f$ follows at once from \cite[Theorem 2.7]{BrC}.
\end{proof}

\section{Examples}\label{Sec:example}

In this last section we provide some examples showing that the hypotheses in Corollary~\ref{eucl-thm} are optimal. 

\begin{example}
Let $D:=\D\times \C$. Note that $D$ is convex, unbounded but not $\C$-proper. Consider the automorphism of $D$ given by $F(z,w)=(z,w+g(z))$, where $g:\D \to \C$ is a holomorphic map which is continuous at no points of $\partial \D$. Hence $F$ does not extend continuously at any point of $\partial D$.
\end{example}

\begin{example}
Let $D=\D\times \D$. Note that $D$ is convex, $\C$-proper but $(D,K_D)$ is not Gromov hyperbolic. Pick points $z_n, w_n \in D$ with $z_n \rightarrow (1,0)$, $w_n \rightarrow (1,1/2)$. Note that 
\begin{align*}
R:=\sup_{n \geq 0} K_D(z_n,w_n) < \infty.
\end{align*}
Then for each integer $n$, pick a small tubular neighborhood $U_n$ of a geodesic joining $z_n$ to $w_n$. By shrinking each $U_n$ and passing to a subsequence we can assume that $\overline{U}_1, \overline{U}_2, \dots$ are all disjoint and the $K_D$-diameter of each $U_n$ is less than $2R$. Now for each $n$ construct a homeomorphism $f_n : U_n \rightarrow U_n$ with $f|_{\partial U_n} = { \rm id}$ where $f_n(z_n) = w_n$ and $f_n(w_n) = z_n$ if $n$ is odd and $f_n(z_n) = z_n$ and $f_n(w_n) = w_n$ if $n$ is even. Let $U = \cup_{n \geq 1} U_n$ and construct a map $f: D \rightarrow D$ where $f|_{D\setminus U} = \id$ and $f|_{U_n} = f_n$. Then $f$ is a $(1,2R)$-quasi-isometry but $f$ does not extend continuously to $\partial D$. 
\end{example}

The previous example is in sharp contrast with the holomorphic case. In fact, if $D\subset \C^d$ is a convex domain and $F:\D^d\to D$ is a biholomorphism, a result of Suffridge \cite{Suf} implies that, up to composition with an affine transformation, $F(z_1,\ldots, z_d)=(f_1(z_1),\ldots, f_d(z_d))$, where $f_j:\D \to \C$ are convex maps. Therefore, $F$ extends as a homeomorphism from $\overline{\D^d}$ to $\overline{D}^\star$.

\medskip

\begin{example}\label{ex}
According to Theorem 1.8 in~\cite{Zi17}  the convex domain 
\begin{align*}
D = \{ (z_0, z) \in \C \times \C^d : { \rm Im}(z_0) > \| z\| \}
\end{align*}
has Gromov hyperbolic Kobayashi metric. By Subsection 1.3 in~\cite{Zi17} the map 
\begin{align*}
f(z_0,z_1,\dots,z_d) = \left( \frac{1}{z_0+i}, \frac{z_1}{z_0+i}, \dots, \frac{z_d}{z_0+i}\right)
\end{align*}
induces a biholomorphism of $D$ onto a bounded $\mathbb C$-convex domain $\Omega$ (which is not convex). Further, the set $\{0\} \times \B_d$ is contained in $\partial \Omega$ and $f^{-1}$ maps this set to  $\{\infty\}$ (in the end compactification of $\overline{D}$). Hence, $f$ does not extend continuously to $\overline{D}^{\star}$.
\end{example}

The domain  $\Omega$ in the previous example is a Jordan domain, namely its boundary is homeomorphic to the Euclidean unit sphere in $\mathbb R^{2(d+1)}$.  Hence, Example~\ref{ex} also shows  the failure  of the Carath\'eodory extension theorem  for (non quasiconformal) univalent maps in complex dimension strictly greater than one. 

We should note that a version of the Carath\'eodory extension theorem  does hold for quasiconformal maps in $\mathbb R^n$ (see, for instance,  \cite[Cor. 6.5.15]{GMP}).

\end{document}